\documentclass{arXiv}

\usepackage{amssymb}
\usepackage{graphicx}
\usepackage[cmtip,all]{xy}

\newtheorem{theorem}{Theorem}[section]
\newtheorem{lemma}[theorem]{Lemma}
\newtheorem{proposition}[theorem]{Proposition}
\newtheorem{corollary}[theorem]{Corollary}

\numberwithin{equation}{section}

\newcommand{\R}{\mathbb{R}}
\newcommand{\Om}{\Omega}
\newcommand{\eps}{\varepsilon}
\newcommand{\ave}[1]{\langle #1\rangle}
\newcommand{\Tmax}{T_{\max}}

\begin{document}

\title[Large-diffusion with exponential nonlinearity]
{Large-diffusion dynamics for a planar Neumann heat equation with exponential nonlinearity}

\author[J. Seo]{Juneyoung Seo}
\address{Department of Mathematics and Institute of Mathematical Science,
Pusan National University, Busan 46241, Republic of Korea}
\email{juneys@pusan.ac.kr}
\thanks{The author was supported by the National Research Foundation of Korea grant funded by the Korea government
(MSIT) (RS-2022-NR072398).}


\date{\today}

\begin{abstract}
We study the Neumann problem $u_t-\eps\Delta u=e^u-1-au\ (a>1)$ on a smooth bounded domain $\Om\subset\R^2$.
For the spatially homogeneous problem, $0$ is stable, the positive
equilibrium $\xi_a$ is unstable, and solutions starting above $\xi_a$ blow up in finite time.
Although finite-time blow-up persists at every diffusivity, 
we show that sufficiently large diffusion recovers this scalar trichotomy uniformly on every bounded $H^1$ ball, and that blow-up occurs precisely when the spatial mean crosses $\xi_a$. 
For initial data with $\|u_0\|_{H^1}\le R$ and spatial mean at
most $\xi_a-\delta$, let $\eps_{\mathrm{unif}}(R,\delta)$ denote the uniform diffusion threshold above which all such solutions are global and converge to $0$. 
We prove $\log \eps_{\mathrm{unif}}(R,\delta)=R^2/(8\pi)+O(\log R)$ as $R\to\infty$.
The domain-independent coefficient $1/(8\pi)$ arises from the sharp
mean-zero Moser--Trudinger inequality. A matching lower bound is obtained
from boundary-concentrating Moser profiles via a localized Kaplan argument.
\end{abstract}

\maketitle

\section{Introduction}\label{sec:intro}

\subsection{Problem and main results}

Let $\Om\subset\R^2$ be a smooth bounded domain, let $a>1$, and
consider
\begin{equation*}
\left\{
\begin{aligned}
 u_t-\eps\Delta u&=f(u),
     &&x\in\Om,\quad t>0,\\
 \partial_\nu u&=0,
     &&x\in\partial\Om,\quad t>0,\\
 u(x,0)&=u_0(x),
     &&x\in\Om,
\end{aligned}
\right.
\tag{$P_\eps$}\label{eq:P}
\end{equation*}
where $f(s)=e^s-1-as$ and $\eps>0$ is the diffusivity. 

The exponential source is of Frank-Kamenetskii type \cite{FrankKamenetskii}, while the
linear term represents distributed heat loss; see McIntosh and Tolputt
\cite{McIntoshTolputt} for a Frank-Kamenetskii model incorporating
volumetric heat losses, including a linear temperature-dependent loss.
The Neumann condition imposes zero conductive heat flux through the boundary.
For the classical combustion background, see Bebernes--Eberly \cite{BebernesEberly};
see Lacey \cite{Lacey} for a PDE analysis of thermal runaway.

For spatially homogeneous states, \eqref{eq:P} reduces to
\begin{equation}\label{eq:ODE}
 y'=f(y).
\end{equation}
The assumption $a>1$ ensures that $f'(0)=1-a<0$ and that the second zero of $f$
lies at a positive level $\xi_a$, producing the stable-origin and unstable-threshold 
scalar phase portrait used throughout the paper. 
Scalar solutions starting below $\xi_a$ converge to $0$,
whereas those starting above $\xi_a$ blow up in finite time. Thus
$\xi_a$ is the scalar threshold between cooling and thermal runaway.

Large diffusion typically suppresses spatial oscillations and drives
reaction--diffusion dynamics toward spatially homogeneous ODE or
finite-dimensional limiting dynamics. Two features make
\eqref{eq:P} different from the standard dissipative setting. First,
finite-time blow-up persists for every diffusivity, so the equation
does not generate a global semiflow on the full phase space. Second,
uniformity over bounded subsets of $H^1(\Om)$ is nontrivial: in two
dimensions such sets admit arbitrarily strong spatial concentration,
which can be amplified by the exponential source before diffusion has
homogenized the solution.

We therefore ask two questions about uniformity in the large-diffusion regime. 
On a prescribed bounded $H^1$ family, can a single diffusivity recover all three
scalar alternatives---decay, threshold behavior, and finite-time
blow-up? If the family is allowed to grow, what is the worst-case
diffusivity required to stabilize all data whose mean remains a fixed
distance below $\xi_a$? 

Both questions are answered in the following theorems.

We write
\[
 \ave{u}=\frac1{|\Om|}\int_\Om u\,dx,
 \qquad
 \|u\|_{H^1}^2=\|u\|_2^2+\|\nabla u\|_2^2,
\]
and let $\Tmax=\Tmax(u_0,\eps)$ denote the maximal existence time.

\begin{theorem}
\label{thm:intro-phase}
For every $R>0$ there exists $\eps_{\mathrm{dyn}}(R)>0$, depending only on
$a$, $\Om$, and $R$, with the following property. If
$\eps>\eps_{\mathrm{dyn}}(R)$ and $\|u_0\|_{H^1}\le R$, then exactly one
of the following alternatives occurs:
\begin{enumerate}
\item the solution blows up in finite time;
\item the solution is global and converges exponentially to $0$ in
      $H^1(\Om)$ and in $L^\infty(\Om)$;
\item the solution is global and converges to the constant equilibrium
      $\xi_a$ in $H^1(\Om)$ and in $L^\infty(\Om)$.
\end{enumerate}
Moreover,
\begin{equation}\label{eq:intro-mean-crossing}
 \Tmax<\infty
 \quad\Longleftrightarrow\quad
 \ave{u(t_0)}>\xi_a
 \quad\text{for some }t_0<\Tmax.
\end{equation}
In particular, within the closed $H^1$ ball of radius $R$, the sets of
initial data giving the first two alternatives are relatively open.
\end{theorem}

The threshold produced by the proof is explicit: it can be taken to satisfy
\begin{equation}\label{eq:intro-dyn-size}
 \eps_{\mathrm{dyn}}(R)
 \le C(1+R)^2
      \exp\left(\frac{R^2}{8\pi}\right)
 \qquad(R\ge1),
\end{equation}
where $C$ depends only on $a$ and $\Om$.

The third alternative is not vacuous: 
Corollary~\ref{cor:threshold-nontrivial} provides 
spatially nonconstant initial data whose solutions converge to $\xi_a$.
The nontrivial direction in \eqref{eq:intro-mean-crossing} is the
exclusion, for sufficiently large diffusion, of concentration-driven
blow-up while the spatial mean remains below $\xi_a$. Indeed, a
supercritical mean forces finite-time blow-up for every diffusivity by
Jensen's inequality. The converse is obtained by combining an a priori
energy barrier with quantitative contraction of spatial oscillations;
compactness and energy dissipation then identify the possible global
limits.

The largeness assumption on the diffusion is essential. 
Calanchi, Ciraolo, and Messina \cite{CCM} recently proved that the stationary Neumann problem
\[
  -\eps\Delta z=f(z)\quad\text{in }\Om,
  \qquad
  \partial_\nu z=0\quad\text{on }\partial\Om
\]
admits nonconstant positive solutions for all small diffusivities. 
Such stationary solutions sit outside the three alternatives of Theorem~\ref{thm:intro-phase}. 

To quantify the diffusion needed uniformly below the scalar threshold,
fix $\delta\in(0,\xi_a)$ and define
\[
 \mathcal A(R,\delta)
 =\left\{u_0\in H^1(\Om):
 \|u_0\|_{H^1}\le R,\ \ave{u_0}\le\xi_a-\delta\right\}.
\]

We say that a diffusivity $\eps>0$ \emph{stabilizes} $\mathcal A(R,\delta)$ if, for every
$u_0\in\mathcal A(R,\delta)$, the solution of \eqref{eq:P} is global and converges to $0$
in $H^1(\Om)$ and in $L^\infty(\Om)$. Since this property is inherited by all larger
diffusivities, we may set
\begin{equation}\label{eq:eps-unif}
 \eps_{\mathrm{unif}}(R,\delta)
 =\inf\bigl\{\eps_0>0:\ \text{every }\eps>\eps_0\text{ stabilizes }\mathcal A(R,\delta)\bigr\}.
\end{equation}

\begin{theorem}
\label{thm:intro-main}
Fix $\delta\in(0,\xi_a)$. Then
\[
 0<\eps_{\mathrm{unif}}(R,\delta)<\infty
 \qquad(R>0).
\]
Moreover, there exist $c_\delta,C_\delta,R_\delta>0$, depending only on
$a$, $\Om$, and $\delta$, such that
\begin{equation}\label{eq:intro-two-sided}
 c_\delta R^{-2}\exp\left(\frac{R^2}{8\pi}\right)
 \le \eps_{\mathrm{unif}}(R,\delta)
 \le C_\delta(1+R)^3\exp\left(\frac{R^2}{8\pi}\right)
\end{equation}
for $R\ge R_\delta$. Consequently,
\[
 \log\eps_{\mathrm{unif}}(R,\delta)
 =\frac{R^2}{8\pi}+O_{a,\Om,\delta}(\log R)
 \qquad(R\to\infty).
\]
\end{theorem}

The coefficient $1/(8\pi)$ is determined by the
Moser--Trudinger inequality. 
For mean-zero functions, the sharp coefficient is $2\pi$
\cite{Cianchi,YangMeanZero}, rather than the classical
Dirichlet coefficient $4\pi$ \cite{Moser}. As a consequence, for
$q\ge1$,
\[
 \|e^u\|_{q}
 \le C(\Om,q)
 \exp\left(
   \ave{u}+\frac{q}{8\pi}\|\nabla u\|_2^2
 \right).
\]
This estimate gives the upper exponential scale in
\eqref{eq:intro-two-sided}. The matching lower bound comes from
boundary-concentrating Moser profiles: the Neumann condition allows
concentration at the boundary, and a localized variant of 
Kaplan's eigenfunction method \cite{Kaplan} converts this concentration 
into finite-time blow-up.
Thus the exponential cost is intrinsic to the dynamics rather than an
artifact of the upper-bound argument.

We will see in Corollary~\ref{cor:gap} that the fixed gap $\delta>0$ cannot be removed:
for $R>\xi_a|\Om|^{1/2}$ one has $\eps_{\mathrm{unif}}(R,\delta)\to\infty$ as $\delta\downarrow0$.

\subsection{Relation to previous work}
\label{sec:intro-context}

Large-diffusion reduction of reaction--diffusion equations has a long
history. Conway, Hoff, and Smoller \cite{ConwayHoffSmoller} obtained 
exponential homogenization toward the associated ODE dynamics 
under a bounded invariant-region hypothesis. 
Related large-diffusivity ODE and shadow-system reductions were developed in \cite{HaleLarge,HaleSakamoto}.
The program was subsequently extended to settings with dispersion and to
parabolic or infinite-dimensional dynamics governed by finite-dimensional
ODEs \cite{CarvalhoODE,CarvalhoHale,CarvalhoPereira}.
Quantitative convergence of attractors in large-diffusion regimes was
obtained in \cite{CarvalhoPires,PiresSamprogna}. In a different direction,
Cupps, Morgan, and Tang \cite{CuppsMorganTang} proved global uniform
boundedness for mass-dissipative reaction--diffusion systems when diffusion
is sufficiently large relative to the initial data. 
For a broader survey on the role of diffusivity in inducing or inhibiting blow-up, 
see Fila and Ninomiya \cite{FilaNinomiya}.

This classical theory largely concerns globally defined dissipative dynamics, 
whereas \eqref{eq:P} also contains finite-time blow-up trajectories. 
Ishige and coauthors \cite{IshigeLargeTime,IshigeMizoguchi,IshigeYagisita} analyzed blow-up times, sets, and profiles for semilinear heat equations with large diffusion. 
Our focus is complementary: we seek a single diffusion regime 
that classifies an entire bounded family while retaining the blow-up branch.

Threshold and borderline dynamics for superlinear heat equations provide
a second comparison; see \cite{QuittnerSouplet} for general background.
For bounded-domain Dirichlet problems, Ni, Sacks, and Tavantzis
\cite{NiSacksTavantzis} analyzed dynamics separating decay from blow-up.
Under certain growth assumptions, every positive ray meets an orbit converging
to a nonzero unstable equilibrium; in a special setting with a unique
positive equilibrium, this yields a decay--equilibrium--blow-up transition.
Related one-parameter threshold asymptotics for whole-space power equations
were studied by Pol\'a\v{c}ik and Quittner
\cite{PolacikQuittnerThreshold} along families $u_0=\alpha\phi$, while
Quittner and Souplet \cite{QuittnerSoupletThreshold} recently analyzed
threshold and subthreshold solutions for bounded-domain Cauchy--Dirichlet
problems and a modified threshold notion in the whole space.
The relevant analogy is the decay--threshold--blow-up structure. 
Here the threshold state is the constant equilibrium $\xi_a$, 
and the classification is sought on whole bounded $H^1(\Om)$ families
rather than along a prescribed one-parameter family.

Exponential reaction--diffusion equations have a separate classical
literature. Fujita \cite{Fujita1969} treated the bounded-domain Dirichlet
problem with exponential reaction, and blow-up for the solid-fuel ignition
model was analyzed in \cite{BebernesBressanEberly}; V\'azquez
\cite{VazquezExp} studied existence and blow-up for the
exponential reaction--diffusion equation. For the pure exponential source
in the whole space, Fujishima \cite{FujishimaExp} obtained an optimal
condition on the decay of the initial data toward $-\infty$, separating
global existence from blow-up. For \eqref{eq:P}, the exponential nonlinearity affects 
the large-diffusion problem through the sharp mean-zero Moser--Trudinger inequality, 
which determines the worst-case cost of controlling concentration on an $H^1$ ball.

Finally, the stationary Neumann problem provides an elliptic counterpart.
Stable nonconstant solutions on convex domains were ruled out in
\cite{CastenHolland,Matano}, while classical large-diffusion nonexistence
results under suitable structural assumptions were obtained by Ni and
Takagi \cite{NiTakagiLarge}. At the small-diffusion end, boundary
concentration and peak location of least-energy solutions were analyzed in
\cite{NT1,NT2}. For the exponential reaction considered here, Calanchi, Ciraolo,
and Messina \cite{CCM} obtain nonconstant positive solutions for
sufficiently small diffusion and conjecture complementary large-diffusion
rigidity; the companion elliptic paper \cite{Seo} proves that rigidity by
a separate elliptic argument. These elliptic results describe the
diffusion-dependent stationary picture, whereas our concern is the
parabolic dynamics in the presence of finite-time blow-up.

\subsection{Strategy of the proof and organization of the paper}
\label{sec:strategy}

\emph{Upper bound.} After the rescaling $s=\eps t$ the reaction has size
$\eps^{-1}$, so on the diffusive time scale the equation is a small
perturbation of the Neumann heat flow. A bootstrap keeps the gradient and
the spatial mean under control long enough for the linear flow to homogenize
the datum; the solution then enters a pointwise strip strictly below
$\xi_a$, where scalar comparison forces exponential decay. The exponential
cost enters through the Moser--Trudinger factor
$\exp\bigl(q\|\nabla u\|_2^2/(8\pi)\bigr)$, and the sharp coefficient is
recovered by letting $q\downarrow1$ at the rate $R^{-2}$.

\emph{Trichotomy.} A mean above $\xi_a$ forces blow-up for every diffusivity
by Jensen's inequality. If the mean never crosses $\xi_a$, an energy barrier
confines the trajectory to $\|\nabla u\|_2^2<R^2+1$; on that region a
quantitative contraction of spatial oscillations makes every
$\omega$-limit point spatially constant, and energy dissipation makes it
stationary. Only $0$ and $\xi_a$ remain.

\emph{Lower bound.} Boundary-concentrating Moser profiles are inserted into
a localized version of Kaplan's eigenfunction method, and the competition
between height and concentration scale is optimized.

The paper is organized as follows. Section~\ref{sec:prelim} develops the
analytical framework. Section~\ref{sec:stab} derives the mean criterion, 
proves uniform stabilization by fast-time homogenization, and records the invariance of the
subthreshold strip. Section~\ref{sec:trichotomy} establishes the
uniform trichotomy, Section~\ref{sec:sharp} identifies the sharp
stabilization scale, and Section~\ref{sec:conclusion} concludes with
extensions and open problems.

\section{Analytical framework}\label{sec:prelim}

We collect the analytical tools used throughout the paper. The two estimates
that determine the critical large-diffusion scale are the Neumann smoothing bound
uniform as $q\downarrow1$ and the endpoint mean-zero Moser--Trudinger inequality. 
The remaining results provide the local solution theory, comparison principle,
and energy-dissipation structure needed to convert these estimates into
dynamical statements.
Unless otherwise stated, $C$ denotes a positive
constant that may depend on $a$ and $\Om$, may change from line to line, and
is independent of $u$, $\eps$, $R$, $\delta$, and the Lebesgue exponent.
Named constants are fixed when they are introduced. We write $u^\perp=u-\ave{u}$ and 
$\|\cdot\|_p=\|\cdot\|_{L^p(\Om)}$.

The Neumann heat semigroup is denoted by $e^{t\Delta_N}$, and
$0=\lambda_1<\lambda_2\le\cdots$ are the eigenvalues of $-\Delta_N$. In particular,
$\|u^\perp\|_2^2\le \lambda_2^{-1}\|\nabla u\|_2^2$ holds for $u\in H^1(\Om)$.

\subsection{The scalar reaction}\label{ssec:scalar}

\begin{lemma}\label{lem:scalar}
Let $f(s)=e^s-1-as$, where $a>1$.
Then $f$ is strictly convex and has precisely two zeros, $0$ and
$\xi_a>\log a$. In addition,
\[
 f>0\quad\text{on }(-\infty,0)\cup(\xi_a,\infty),
 \qquad
 f<0\quad\text{on }(0,\xi_a),
\]
and $\int_c^\infty ds/f(s)<\infty$ for $c>\xi_a$.
For $M<\xi_a$, define
\[
 c_M=
 \begin{cases}
  a-1,&M\le0,\\
  \displaystyle a-\frac{e^M-1}{M},&0<M<\xi_a.
 \end{cases}
\]
Then $c_M>0$ and
\begin{equation}\label{eq:sfs}
 sf(s)\le -c_Ms^2
 \qquad(s\le M).
\end{equation}
\end{lemma}

\begin{proof}
Strict convexity, $f(0)=0$, $f'(0)=1-a<0$, and $f(s)\to\infty$ as $s\to\infty$ give exactly
one positive zero.  Since
$f(\log a)=a-1-a\log a<0$, this zero satisfies $\xi_a>\log a$; the sign
properties and the Osgood integral then follow immediately.  If $s\le0$,
$e^s\ge1+s$ gives $sf(s)\le-(a-1)s^2$.  If $0<s\le M<\xi_a$, monotonicity
of $(e^s-1)/s$ gives $f(s)/s\le-c_M$, proving
\eqref{eq:sfs}.
\end{proof}

\subsection{Neumann smoothing}\label{ssec:smoothing}

\begin{lemma}\label{lem:uniform-smoothing}
There exists $C_{\mathrm{sm}}\ge1$ such that, for every $q\in[1,2]$ and
$t>0$,
\begin{equation}\label{eq:Lq-semigroup}
 \|e^{t\Delta_N}g\|_\infty
 +\|\nabla e^{t\Delta_N}g\|_2
 \le C_{\mathrm{sm}}(1+t^{-1/q})\|g\|_q,
\end{equation}
where $C_{\mathrm{sm}}$ is independent of $q$. In particular,
\begin{equation}\label{eq:semigroup}
 \|e^{t\Delta_N}g\|_\infty
 +\|e^{t\Delta_N}g\|_{H^1}
 \le C_{\mathrm{sm}}(1+t^{-1/2})\|g\|_2
 \qquad(t>0),
\end{equation}
and
\begin{equation}\label{eq:gradient-contractivity}
 \|\nabla e^{t\Delta_N}u\|_2\le\|\nabla u\|_2
 \qquad(u\in H^1(\Om)).
\end{equation}
\end{lemma}

\begin{proof}
Gaussian Neumann heat-kernel bounds, duality, and the spectral theorem give
\[
 \|e^{t\Delta_N}\|_{1\to\infty}\le C(1+t^{-1}),\quad
 \|e^{t\Delta_N}\|_{2\to\infty}\le C(1+t^{-1/2}),
\]
\[
 \|e^{t\Delta_N}\|_{1\to2}\le C(1+t^{-1/2}),\quad
 \|\nabla e^{t\Delta_N}\|_{2\to2}\le Ct^{-1/2};
\]
see \cite[Chapter~6]{Ouhabaz}.  Interpolation of the first pair yields
$\|e^{t\Delta_N}\|_{q\to\infty}\le C(1+t^{-1/q})$ with $C$ uniform for
$q\in[1,2]$.  For the gradient, factor the semigroup at time $t/2$ and
interpolate $e^{(t/2)\Delta_N}:L^q\to L^2$ between $L^1\to L^2$ and
$L^2$ contractivity.  The resulting time factor is $O(t^{-1/q})$ for
$t\le1$ and $O(1)$ for $t\ge1$, again uniformly in $q$.  This proves
\eqref{eq:Lq-semigroup}.  The remaining assertions follow from the case
$q=2$, $L^2$ contractivity, and the Neumann spectral resolution.
\end{proof}

\subsection{Moser--Trudinger estimates}\label{ssec:MT}

We use the sharp Moser--Trudinger inequality for mean-zero functions on
bounded smooth planar domains; see, for example, \cite{ChangYang,Cianchi,YangMeanZero}. There exists
$C_{\mathrm{MT}}=C_{\mathrm{MT}}(\Om)\ge1$ such that
\begin{equation}\label{eq:MT}
 \frac1{|\Om|}\int_\Om
 \exp\left(2\pi\frac{v^2}{\|\nabla v\|_2^2}\right)\,dx
 \le C_{\mathrm{MT}}
\end{equation}
whenever $v\in H^1(\Om)$, $\ave{v}=0$, and $v\not\equiv0$.

\begin{lemma}\label{lem:expint}
If $v\in H^1(\Om)$ and $\ave{v}=0$, then for every $\beta>0$,
\begin{equation}\label{eq:expint}
 \int_\Om e^{\beta|v|}\,dx
 \le C_{\mathrm{MT}}|\Om|
 \exp\left(\frac{\beta^2}{8\pi}\|\nabla v\|_2^2\right).
\end{equation}
Consequently, for $q\in[1,\infty)$ and $u\in H^1(\Om)$,
\begin{equation}\label{eq:Lq}
 \|e^u\|_{q}
 \le(C_{\mathrm{MT}}|\Om|)^{1/q}
 \exp\left(\ave{u}+\frac q{8\pi}\|\nabla u\|_2^2\right).
\end{equation}
\end{lemma}

\begin{proof}
If $v\not\equiv0$, then
$\beta|v|\le 2\pi v^2/\|\nabla v\|_2^2
+\frac{\beta^2}{8\pi}\|\nabla v\|_2^2$.
Apply \eqref{eq:MT}. The case $v\equiv0$ is immediate. To obtain
\eqref{eq:Lq}, write $u=\ave{u}+u^\perp$, apply \eqref{eq:expint} with
$\beta=q$, and take the $q$-th root.
\end{proof}

\begin{lemma}\label{lem:Lr-gradient}
There exists $C>0$, depending only on $\Om$, such that
every $v\in H^1(\Om)$ with $\ave{v}=0$ satisfies
\[
 \|v\|_r^2
 \le Cr\|\nabla v\|_2^2
 \qquad(r\ge2).
\]
\end{lemma}

\begin{proof}
For $v\not\equiv0$, retain the $k$-th term in the expansion of \eqref{eq:MT}:
\[
 \frac{(2\pi)^k}{k!\,\|\nabla v\|_2^{2k}}\int_\Om|v|^{2k}\,dx
 \le C_{\mathrm{MT}}|\Om|.
\]
Thus $\|v\|_{2k}^2\le Ck\|\nabla v\|_2^2$.  Given $r\ge2$, take
$k=\lceil r/2\rceil$ and use the embedding
$L^{2k}(\Om)\hookrightarrow L^r(\Om)$.  The case $v\equiv0$ is immediate.
\end{proof}

\begin{lemma}\label{lem:lip}
The map $u\mapsto e^u$ is locally Lipschitz from $H^1(\Om)$ to
$L^2(\Om)$. More precisely, for every $R>0$ there exists $C(R)>0$
such that
\[
 \|e^u-e^w\|_2
 \le C(R)\|u-w\|_{H^1}
\]
whenever $\|u\|_{H^1},\|w\|_{H^1}\le R$.
\end{lemma}

\begin{proof}
The pointwise mean-value estimate gives
$|e^u-e^w|\le |u-w|(e^u+e^w)$.
Therefore, by H\"older's inequality and the embedding
$H^1(\Om)\hookrightarrow L^4(\Om)$,
\[
 \|e^u-e^w\|_2
 \le \|u-w\|_4(\|e^u\|_4+\|e^w\|_4)
 \le C(\|e^u\|_4+\|e^w\|_4)\|u-w\|_{H^1}.
\]
The exponential factors are uniformly bounded on the $H^1$ ball by
\eqref{eq:Lq} with $q=4$.
\end{proof}

\subsection{Local well-posedness, regularization, and comparison}\label{ssec:lwp}

The initial phase space is $H^1(\Om)$ rather than $L^\infty(\Om)$.
Local $H^1$ theories for two-dimensional exponential heat equations have
been developed, for example, by Ibrahim--Jrad--Majdoub--Saanouni
\cite{IbrahimJradMajdoubSaanouni}.  In the present bounded Neumann setting,
the preceding exponential-integrability estimates make the nonlinearity
locally Lipschitz from $H^1$ to $L^2$, so standard analytic-semigroup theory
applies directly.  Positive-time smoothing then restores the pointwise
regularity needed for comparison and, later, for the phase-space arguments.

Let $A_\eps=I-\eps\Delta_N$. Then with equivalent norms,
\[
 \mathcal D(A_\eps^{1/2})=H^1(\Om),
 \quad
 \mathcal D(A_\eps)=H^2_N(\Om):=\{u\in H^2(\Om):\partial_\nu u=0\}.
\]
Writing $\mathcal G(u)=f(u)+u=e^u-1-(a-1)u$, problem \eqref{eq:P} takes the form
\[
 u_t+A_\eps u=\mathcal G(u).
\]
By Lemma~\ref{lem:lip}, $\mathcal G:H^1(\Om)\to L^2(\Om)$ is locally
Lipschitz.
Whenever the solution exists up to time $t$, we write
$\mathcal S_\eps(t)u_0=u(t;u_0)$;
when $\eps$ is fixed, the subscript is omitted.

\begin{proposition}\label{prop:lwp}
For every $u_0\in H^1(\Om)$ and $\eps>0$, there is a unique maximal mild
solution $u\in C([0,\Tmax);H^1(\Om))$, which satisfies
\begin{equation}\label{eq:mild}
 u(t)=e^{\eps t\Delta_N}u_0
 +\int_0^t e^{\eps(t-s)\Delta_N}f(u(s))\,ds.
\end{equation}
For every $0<\tau<T<\Tmax$,
\begin{equation}\label{eq:positive-strong}
 u\in C([\tau,T];H_N^2(\Om))\cap C^1([\tau,T];L^2(\Om)).
\end{equation}
Solutions depend continuously on their initial data in
$C([0,T];H^1(\Om))$ on every interval strictly below the maximal existence
time, and
\begin{align}
 \Tmax<\infty
 &\Longrightarrow\limsup_{t\uparrow\Tmax}\|u(t)\|_{H^1}=\infty,
 \label{eq:blowup-alt}\\
 \Tmax<\infty
 &\Longrightarrow\limsup_{t\uparrow\Tmax}\|u(t)\|_\infty=\infty.
 \label{eq:Linf-continuation}
\end{align}
Moreover, if $0<T<\Tmax(u_0)$, then on a sufficiently small $H^1$
neighborhood $U$ of $u_0$ the time-$T$ map
$\mathcal S_\eps(T):U\to H^1(\Om)\cap L^\infty(\Om)$
is continuous,
where the target is endowed with the norm
$\|\cdot\|_{H^1}+\|\cdot\|_\infty$.
\end{proposition}

\begin{proof}
By Lemma~\ref{lem:lip} and the above operator identifications,
the standard semilinear theory \cite[Sections~3.3--3.4]{Henry} applies to
$u_t+A_\eps u=e^u-1-(a-1)u$, giving existence, uniqueness, continuous
dependence, and the $H^1$ continuation criterion.  Analytic-semigroup
regularization gives \eqref{eq:positive-strong}; see, for example,
\cite[Section~7.1]{Lunardi}.

If $u$ were bounded in $L^\infty$ near a finite $\Tmax$, then $f(u)$ would
be bounded in $L^2$ there; restarting \eqref{eq:mild} and using
$L^2\to H^1$ smoothing would give a uniform $H^1$ bound, contradicting
\eqref{eq:blowup-alt}. This proves \eqref{eq:Linf-continuation}.
Finally, if $u_{0,n}\to u_0$ in $H^1$ and all solutions are considered up
to a fixed $T<\Tmax(u_0)$, continuous dependence gives a common $H^1$ bound.
Restarting the difference of the mild formulas at $T/2$, Lemma~\ref{lem:lip}
and the $L^2\to L^\infty$ estimate imply
\[
 \|u_n(T)-u(T)\|_\infty
 \le C_{\eps,T,R}\sup_{0\le t\le T}\|u_n(t)-u(t)\|_{H^1}\to0.
\]
The $H^1$ convergence is already part of continuous dependence.
\end{proof}

Throughout the paper, finite-time blow-up means $\Tmax<\infty$ for the
maximal $H^1$ solution above.

\begin{lemma}\label{lem:comparison}
If $u_0,w_0\in H^1(\Om)$ and $u_0\le w_0$ almost everywhere, then the
corresponding solutions remain ordered throughout their common existence
interval. The same conclusion holds when one of the two functions is a
spatially constant classical subsolution or supersolution.
\end{lemma}

\begin{proof}
For bounded initial data, the mild formula and positive-time smoothing make
two solutions uniformly bounded on every compact common existence interval.
Replace $f$ outside their common range by a globally Lipschitz function and
test the difference equation by its positive part; Gronwall's inequality
preserves the order. For general $H^1$ data, apply this argument to the
ordered truncations $\max\{-n,\min\{u_0,n\}\}$ and pass to the limit using
Proposition~\ref{prop:lwp}. The same positive-part argument after any
positive restart time applies when one of the comparison functions is a spatially constant
classical sub- or supersolution; letting the restart time tend to the
comparison time gives the stated conclusion.
\end{proof}

\subsection{Mean evolution and energy dissipation}\label{ssec:mean-energy}

Two scalar quantities will organize the global dynamics. Convexity of the
reaction makes the spatial mean dominate the scalar ODE through Jensen's
inequality, while the natural parabolic energy is decreasing along strong
trajectories. The first gives diffusion-independent blow-up criteria; the
second supplies the bounds used in Section~\ref{sec:trichotomy}.

\begin{lemma}\label{lem:mean}
Let $m(t)=\ave{u(t)}$. Then $m\in C^1([0,\Tmax))$ and
\[
 m'(t)=\ave{f(u(t))}\ge f(m(t)).
\]
Consequently, if $y$ solves \eqref{eq:ODE} with $y(t_0)=m(t_0)$, then
\[
 m(t)\ge y(t)
\]
for $t\ge t_0$, as long as both functions are finite.
\end{lemma}

\begin{proof}
Taking the mean in \eqref{eq:mild} gives
$m(t)=m(0)+\int_0^t\ave{f(u(s))}\,ds$.
The local Lipschitz property of $f:H^1(\Om)\to L^2(\Om)$ implies continuity
of the integrand, including at $t=0$. 
Jensen's inequality and scalar comparison end the proof.
\end{proof}

Define the primitive and the associated energy functional by
\begin{align*}
     &F(s)=\int_0^s f(\sigma)\,d\sigma
 =e^s-1-s-\frac a2s^2,\\
 &E_\eps(u)
 =\frac\eps2\|\nabla u\|_2^2
 -\int_\Om F(u)\,dx,
 \qquad u\in H^1(\Om).
\end{align*}

The functional is finite and continuous on $H^1(\Om)$ by Lemma~\ref{lem:lip}.

\begin{lemma}\label{lem:energy}
For every solution of \eqref{eq:P} and every
$0\le t_0< t_1<\Tmax$,
\begin{equation}\label{eq:energy-identity}
 E_\eps(u(t))
 +\int_{t_0}^t\|u_t(\tau)\|_2^2\,d\tau
 =E_\eps(u(t_0)),\quad t_0<t<t_1.
\end{equation}
In particular, $E_\eps(u(t))$ is nonincreasing on the maximal existence
interval.
\end{lemma}

\begin{proof}
For $t_0<t<t_1$, Proposition~\ref{prop:lwp} gives the regularity needed
for the standard chain rule and integration by parts. Hence
\[
 \frac d{dt}E_\eps(u(t))
 =\int_{\Om}(-\eps\Delta u-f(u))u_t\, dx
 =-\|u_t\|_2^2
\]
for $t\in(t_0,t_1)$. Integration proves \eqref{eq:energy-identity} for
$t_0>0$.  Letting $t_0\downarrow0$ and using $u(t_0)\to u_0$ in $H^1$ together
with continuity of $E_\eps:H^1\to\R$ proves the identity at $t_0=0$.
\end{proof}

\section{Mean criterion and quantitative stabilization}\label{sec:stab}

This section establishes the diffusion-independent mean criterion and then
develops the fast-time homogenization mechanism. The resulting uniform stabilization theorem will be used both in the phase-portrait analysis of Section~\ref{sec:trichotomy} and, after optimization of its parameters, in the upper bound for the stabilization threshold in Section~\ref{sec:sharp}.

\subsection{The mean criterion}

\begin{proposition}\label{prop:mean-threshold}
Let $u$ be the maximal $H^1$ solution of \eqref{eq:P}.
\begin{enumerate}
\item If $\ave{u_0}>\xi_a$, then for every $\eps>0$,
\begin{equation}\label{eq:mean-blowup-time}
 \Tmax\le\int_{\ave{u_0}}^\infty\frac{ds}{f(s)}<\infty.
\end{equation}
\item If $\ave{u_0}=\xi_a$, then either $u_0\equiv\xi_a$ and
$u\equiv\xi_a$, or $\Tmax<\infty$ for every $\eps>0$.
\end{enumerate}
\end{proposition}

\begin{proof}
For the first assertion, let $y'=f(y)$ with $y(0)=\ave{u_0}$.  By
Lemma~\ref{lem:mean}, $\ave{u(t)}\ge y(t)$ while both are finite, and the
scalar Osgood integral in \eqref{eq:mean-blowup-time} gives the claimed
upper bound for $\Tmax$. For the second, if $u_0$ is nonconstant, strict Jensen gives
$\left.\frac d{dt}\ave{u(t)}\right|_{t=0}
=\ave{e^{u_0}}-e^{\xi_a}>0$.
The mean is therefore supercritical at a small positive time, and the first
assertion applies after restarting.
\end{proof}

Proposition~\ref{prop:mean-threshold} settles the supercritical and
critical mean independently of diffusion. We now show that a fixed
subcritical mean gap leads, uniformly over bounded $H^1$ families, to
the opposite behavior when diffusion is sufficiently strong.

\subsection{Homogenization on the diffusive time scale}\label{ssec:fast-time}

The stabilization mechanism acts on the diffusive time scale. After the
rescaling $s=\eps t$, the equation becomes the Neumann heat flow perturbed by
a reaction of size $\eps^{-1}$. The purpose of the next estimates is to keep
the gradient and mean under uniform control throughout a fast-time interval
long enough for the linear heat flow to homogenize the initial profile.

Fix $R>0$ and $\delta\in(0,\xi_a)$, and assume
\begin{equation}\label{eq:data}
 \|u_0\|_{H^1}\le R,
 \qquad
 \ave{u_0}\le\xi_a-\delta.
\end{equation}
Set $D_R=R^2+1$, $d_R=\sqrt{D_R}-R$, and
$M_R=R/\sqrt{|\Om|}+1/2$.
Here $D_R$ is the bootstrap bound for $\|\nabla v\|_2^2$,
$d_R$ is the corresponding gap above the initial gradient
bound, and $M_R$ bounds the absolute value of the mean. The additive choice $D_R=R^2+1$ preserves the sharp exponential scale while ensuring that $d_R\asymp R^{-1}$.
For $q\in(1,2]$, define
\begin{equation}\label{eq:FqR}
\begin{aligned}
 \mathcal F_{q,R}={}&(C_{\mathrm{MT}}|\Om|)^{1/q}
 \exp\left(\xi_a+\frac12+\frac{qD_R}{8\pi}\right)+|\Om|^{1/q}\\
 &+a\left(
     |\Om|^{1/q}M_R
     +|\Om|^{1/q-1/2}\sqrt{\frac{D_R}{\lambda_2}}
    \right),
\end{aligned}
\end{equation}
and set
\[
\mathcal H_q(S)
   =S+\frac{q}{q-1}S^{1-1/q},
   \qquad S>0.
\]
Here $\mathcal F_{q,R}$ and $H_q(S)$ will serve as a uniform $L^q$-bound during the bootstrap.

Introduce the fast time $s=\eps t$ and set $\tilde{u}(x,s)=u(x,s/\eps)$.
\begin{equation*}
\left\{
\begin{aligned}
 \tilde u_s-\Delta\tilde u&=\eps^{-1}f(\tilde u), && x\in\Om,\ 0<s<S_{\max},\\
 \partial_\nu\tilde u&=0, && x\in\partial\Om,\ 0<s<S_{\max},\\
 \tilde u(x,0)&=u_0(x), && x\in\Om,
\end{aligned}
\right.
\tag{$\tilde P_\eps$}\label{eq:Ptilde}
\end{equation*}
where $S_{\max}=\eps\Tmax$.

\begin{lemma}\label{lem:barrier}
Assume \eqref{eq:data}. Let $q\in(1,2]$ and $S>0$. If
\[
 \eps>
 \max\left\{
  \frac{C_{\mathrm{sm}}\mathcal{H}_q(S)\mathcal F_{q,R}}{d_R},
  \ 2S|\Om|^{-1/q}\mathcal F_{q,R}
 \right\},
\]
then $S_{\max}>S$ and, for $0\le s\le S$,
\begin{equation}\label{eq:barrier}
 \|\nabla \tilde{u}(s)\|_2^2\le D_R,
 \qquad
 |\ave{\tilde{u}(s)}-\ave{u_0}|\le\frac12,
 \qquad
 \|f(\tilde{u}(s))\|_q\le\mathcal F_{q,R}.
\end{equation}
Furthermore,
\begin{equation}\label{eq:heatapprox-Linf}
 \|\tilde{u}(S)-e^{S\Delta_N}u_0\|_\infty
 \le
 \frac{C_{\mathrm{sm}}\mathcal{H}_q(S)\mathcal F_{q,R}}{\eps}.
\end{equation}
\end{lemma}

\begin{proof}
The rescaled mild formula is
\begin{equation}\label{eq:fast-mild}
 \tilde{u}(s)=e^{s\Delta_N}u_0
 +\frac1\eps\int_0^s e^{(s-r)\Delta_N}f(\tilde{u}(r))\,dr.
\end{equation}
Let $\sigma_*$ be the supremum of all
$\sigma<\min\{S,S_{\max}\}$ such that
$\|\nabla \tilde{u}(r)\|_2<\sqrt{D_R}$ and
$|\ave{\tilde{u}(r)}-\ave{u_0}|<1/2$ for $0\le r\le\sigma$.
Since $\|\nabla u_0\|_2\le R<\sqrt{D_R}$, this set is nonempty.

For $0\le s<\sigma_*$, the bootstrap assumptions and \eqref{eq:data} give
$\ave{\tilde{u}(s)}\le\ave{u_0}+1/2
\le\xi_a-\delta+1/2\le\xi_a+1/2$.
Also,
$|\ave{u_0}|\le|\Om|^{-1/2}\|u_0\|_2
\le R/\sqrt{|\Om|}$,
so $|\ave{\tilde{u}(s)}|\le M_R$. Together with
$\|\nabla \tilde{u}(s)\|_2^2\le D_R$, Lemma~\ref{lem:expint} gives
\[
 \|e^{\tilde{u}(s)}\|_q
 \le(C_{\mathrm{MT}}|\Om|)^{1/q}
 \exp\left(\xi_a+\frac12+\frac{qD_R}{8\pi}\right).
\]
Since $q\le2$, Poincar\'e inequality and the embedding
$L^2(\Om)\hookrightarrow L^q(\Om)$ yield
\[
 \|\tilde{u}(s)\|_q
 \le |\Om|^{1/q}|\ave{\tilde{u}(s)}|+\|\tilde{u}(s)^\perp\|_q
 \le |\Om|^{1/q}M_R
 +|\Om|^{1/q-1/2}\sqrt{\frac{D_R}{\lambda_2}}.
\]
Hence $\|f(\tilde{u}(s))\|_q\le\mathcal F_{q,R}$ for $0\le s<\sigma_*$.
Taking the gradient in \eqref{eq:fast-mild}, using
\eqref{eq:gradient-contractivity} and \eqref{eq:Lq-semigroup}, gives
\[
 \begin{aligned}
 \|\nabla \tilde{u}(s)\|_2
 &\le R+\frac{C_{\mathrm{sm}}}{\eps}
   \int_0^s\bigl(1+(s-r)^{-1/q}\bigr)\|f(\tilde{u}(r))\|_q\,dr\\
 &\le R+
 \frac{C_{\mathrm{sm}}\mathcal{H}_q(S)\mathcal F_{q,R}}{\eps}
 <\sqrt{D_R}.
 \end{aligned}
\]
Taking the spatial mean in \eqref{eq:fast-mild} gives
$\ave{\tilde{u}(s)}-\ave{u_0}
=\eps^{-1}\int_0^s\ave{f(\tilde{u}(r))}\,dr$,
and therefore
$|\ave{\tilde{u}(s)}-\ave{u_0}|
\le S|\Om|^{-1/q}\mathcal F_{q,R}/\eps<1/2$.
Both bootstrap inequalities improve strictly. Suppose, for contradiction,
that $\sigma_*<\min\{S,S_{\max}\}$. Then $\tilde{u}(\sigma_*)$ is defined,
and passing to $s\uparrow\sigma_*$ in the
improved estimates shows, by continuity, that the strict bounds still hold
at $s=\sigma_*$.  They therefore persist on a slightly larger interval,
contradicting the definition of $\sigma_*$.  Hence
$\sigma_*=\min\{S,S_{\max}\}$.

If $S_{\max}\le S$, the preceding bounds and Poincar\'e inequality give
\[
 \sup_{0\le s<S_{\max}}\|\tilde{u}(s)\|_{H^1}<\infty.
\]
This contradicts the continuation criterion. Hence $S_{\max}>S$, and
\eqref{eq:barrier} follows by continuity.

Finally, evaluate the Duhamel term in \eqref{eq:fast-mild} at $S$ and use
the $L^q$--$L^\infty$ estimate:
\[
 \|\tilde{u}(S)-e^{S\Delta_N}u_0\|_\infty
 \le\frac{C_{\mathrm{sm}}\mathcal F_{q,R}}\eps
 \int_0^S\bigl(1+(S-r)^{-1/q}\bigr)\,dr
 =\frac{C_{\mathrm{sm}}\mathcal{H}_q(S)\mathcal F_{q,R}}\eps.
\]
This proves \eqref{eq:heatapprox-Linf}.
\end{proof}

\begin{lemma}\label{lem:heat-homog}
There exists $C_\Om>0$ such that
\begin{equation}\label{eq:heat-homog}
 \|e^{S\Delta_N}g-\ave{g}\|_\infty
 \le C_\Om e^{-\lambda_2S/2}\|g^\perp\|_2
 \qquad(S\ge1).
\end{equation}
\end{lemma}

\begin{proof}
Since the Neumann heat semigroup preserves constants,
$e^{S\Delta_N}g-\ave{g}=e^{S\Delta_N}g^\perp$.
Factor at time $S/2$ and combine \eqref{eq:semigroup} with the spectral gap:
\[
 \begin{aligned}
 \|e^{S\Delta_N}g-\ave{g}\|_\infty
 &\le C_{\mathrm{sm}}\left(1+(S/2)^{-1/2}\right)
       \|e^{(S/2)\Delta_N}g^\perp\|_2\\
 &\le(1+\sqrt{2})C_{\mathrm{sm}}
       e^{-\lambda_2S/2}\|g^\perp\|_2.
 \end{aligned}
\]
Thus \eqref{eq:heat-homog} holds, for example, with
$C_\Om=(1+\sqrt{2})C_{\mathrm{sm}}$.
\end{proof}

Define
\begin{align}
S_{R,\delta}
&=1+\frac{2}{\lambda_2}\log_+\left(\frac{4C_\Om R}{\delta}\right),
\label{eq:S-R-delta}
\\
\eps_{\mathrm{stab}}^{(q)}(R,\delta)
&=1+
C_{\mathrm{sm}}
  \mathcal H_q(S_{R,\delta})
  \mathcal F_{q,R}\left(d_R^{-1}+\frac4\delta\right)+
 2S_{R,\delta}|\Om|^{-1/q}\mathcal F_{q,R}.
\label{eq:eps-stab}
\end{align}
Here, $\log_+x=\max\{0,\log x\}$.

The $d_R^{-1}$ contribution in \eqref{eq:eps-stab}, together with the last term,
guarantees the bootstrap conditions of Lemma~\ref{lem:barrier}. 
The $4/\delta$ contribution then makes the nonlinear Duhamel error smaller than $\delta/4$. Set $t_\eps=t_\eps(R,\delta)=S_{R,\delta}/\eps$.

\begin{proposition}\label{prop:entry}
If \eqref{eq:data} holds and
$\eps>\eps_{\mathrm{stab}}^{(q)}(R,\delta)$
for some $q\in(1,2]$, then the solution exists up to $t_\eps$ and
\begin{equation}\label{eq:entry}
 \|u(t_\eps)-\ave{u_0}\|_\infty\le\frac\delta2.
\end{equation}
Consequently,
\begin{equation}\label{eq:entry-strip}
 -\frac R{\sqrt{|\Om|}}-\frac\delta2
 \le u(x,t_\eps)
 \le\xi_a-\frac\delta2.
\end{equation}
\end{proposition}

\begin{proof}
Apply Lemma~\ref{lem:barrier} with $S=S_{R,\delta}$. Since
$\|u_0^\perp\|_2\le R$, the definition of $S_{R,\delta}$ and
Lemma~\ref{lem:heat-homog} give
$\|e^{S_{R,\delta}\Delta_N}u_0-\ave{u_0}\|_\infty\le\delta/4$.
The first term in \eqref{eq:eps-stab} and
\eqref{eq:heatapprox-Linf} give $\|\tilde u(S_{R,\delta})-e^{S_{R,\delta}\Delta_N}u_0\|_\infty
 <\delta/4$.
Since $u(t_\eps)=\tilde u(S_{R,\delta})$, the triangle inequality proves
\eqref{eq:entry}. Since
$-R/\sqrt{|\Om|}\le\ave{u_0}\le\xi_a-\delta$,
\eqref{eq:entry-strip} follows.
\end{proof}

\subsection{Invariant subthreshold dynamics}\label{ssec:invariance}

The fast-time argument is needed only to enter a pointwise region strictly
below $\xi_a$. Once such entry has occurred, scalar comparison takes over:
the subthreshold strip is invariant and its upper and lower scalar
envelopes both decay to zero.

\begin{lemma}\label{lem:invariance}
Suppose that, at some $t_1<\Tmax$,
$-\mu\le u(t_1)\le M$, where $\mu\ge0$ and $0\le M<\xi_a$.
Then $\Tmax=\infty$, these pointwise bounds remain invariant,
\begin{equation}\label{eq:invariant-strip}
 -\mu\le u(\cdot,t)\le M
 \qquad\text{a.e.\ in }\Om,\quad t\ge t_1,
\end{equation}
and
\begin{equation}\label{eq:envelope}
 \|u(t)\|_\infty
 \le\max\{M,\mu\}e^{-c_M(t-t_1)}
 \qquad(t\ge t_1).
\end{equation}
\end{lemma}

\begin{proof}
Let $\overline y$ and $\underline y$ solve \eqref{eq:ODE} with
$\overline y(t_1)=M$ and $\underline y(t_1)=-\mu$.
By Lemma~\ref{lem:scalar}, both are global and satisfy
$-\mu\le\underline y(t)\le0\le\overline y(t)\le M$.
Lemma~\ref{lem:comparison} gives
$\underline y(t)\le u(x,t)\le\overline y(t)$
while the PDE exists. 

For either scalar orbit, \eqref{eq:sfs} gives
\[
 \frac12\frac{d}{dt}|y|^2=y f(y)\le -c_M|y|^2 ,
\]
whence $|\underline y(t)|\le\mu e^{-c_M(t-t_1)}$ and $|\overline y(t)|\le Me^{-c_M(t-t_1)}$.
This proves \eqref{eq:envelope}. The pointwise comparison bounds and
\eqref{eq:Linf-continuation} exclude finite-time breakdown.
\end{proof}

\begin{theorem}\label{thm:stab}
Let $q\in(1,2]$ and suppose that
$\eps>\eps_{\mathrm{stab}}^{(q)}(R,\delta)$.
Every initial datum satisfying \eqref{eq:data} generates a global solution
converging exponentially to $0$ in $L^\infty(\Om)$ and $H^1(\Om)$. More precisely, with
$M=\xi_a-\delta/2$, $\mu=R/\sqrt{|\Om|}+\delta/2$, and
$B=\max\{M,\mu\}$, one has
\begin{align}
 \|u(t)\|_\infty
 &\le B e^{-c_M(t-t_\eps)}
 &&(t\ge t_\eps),
 \label{eq:stab-Linf}\\
 \|u(t)\|_{H^1}
 &\le C(R,\delta)e^{-c_M(t-t_\eps)}
 &&(t\ge t_\eps+1).
 \label{eq:stab-H1}
\end{align}
The constant in \eqref{eq:stab-H1} is independent of the admissible datum
and of $\eps>\eps_{\mathrm{stab}}^{(q)}(R,\delta)$.
\end{theorem}

\begin{proof}
Proposition~\ref{prop:entry} and Lemma~\ref{lem:invariance} give global
existence, \eqref{eq:invariant-strip}, and \eqref{eq:stab-Linf}. Since
$f(0)=0$ and
$L_M=\sup_{s\le M}|f'(s)|<\infty$,
one has $|f(s)|\le L_M|s|$ for $s\le M$. In particular, the upper bound
in \eqref{eq:invariant-strip} makes this estimate applicable pointwise to
$u(t-r)$. For $t\ge t_\eps+1$, restart
the mild formula at $t-1$. Since $\eps>1$,
$1+(\eps r)^{-1/2}\le1+r^{-1/2}$ for $0<r\le1$.
Using \eqref{eq:semigroup} and \eqref{eq:stab-Linf},
\[
 \begin{aligned}
 \|u(t)\|_{H^1}
 &\le 2C_{\mathrm{sm}}\|u(t-1)\|_2+C_{\mathrm{sm}}\int_0^1
       (1+r^{-1/2})\|f(u(t-r))\|_2\,dr\\
 &\le C(R,\delta)e^{-c_M(t-t_\eps)}.
 \end{aligned}
\]
This proves \eqref{eq:stab-H1}.
\end{proof}

In particular,
\[
 \eps_{\mathrm{unif}}(R,\delta)
 \le
 \inf_{1<q\le2}\eps_{\mathrm{stab}}^{(q)}(R,\delta).
\]
The choice of $q$ will be optimized in Section~\ref{sec:sharp}.

The preceding theorem also clarifies why a fixed gap below the scalar
threshold is necessary for uniform stabilization.

\begin{corollary}\label{cor:gap}
If $R>\xi_a|\Om|^{1/2}$, then
\[
 \eps_{\mathrm{unif}}(R,\delta)\longrightarrow\infty
 \qquad(\delta\downarrow0).
\]
\end{corollary}

\begin{proof}
Choose a nonzero smooth $\phi$ with $\ave{\phi}=0$ and
$\|\xi_a+\phi\|_{H^1}<R$. Fix $\eps_0>0$. At diffusivity $\eps_0$, the
solution with initial datum $\xi_a+\phi$ has critical mean and is
nonconstant, so Proposition~\ref{prop:mean-threshold} shows that its
mean exceeds $\xi_a$ at some time $t_{*}>0$. By continuous
dependence, the same strict crossing holds for the solution from
$u_{0,\delta}=\xi_a-\delta+\phi$ when $\delta>0$ is sufficiently small.
Moreover, $\|u_{0,\delta}\|_{H^1}<R$ and $\ave{u_{0,\delta}}=\xi_a-\delta$, 
so $u_{0,\delta}\in\mathcal A(R,\delta)$.
Thus stabilization fails at $\eps=\eps_0$, so
$\eps_{\mathrm{unif}}(R,\delta)\ge\eps_0$ for all sufficiently small
$\delta$. Since $\eps_0$ is arbitrary, the conclusion follows.
\end{proof}

\section{Oscillation contraction and the global trichotomy}\label{sec:trichotomy}

We turn from the quantitative stabilization framework of
Section~\ref{sec:stab} to the global dynamics on a bounded $H^1$ ball. The
aim is to show that sufficiently large diffusion leaves no additional bounded
asymptotic behavior beyond that of the scalar equation. The two complementary
ingredients are a quantitative contraction of spatial oscillations and
an energy barrier for trajectories whose mean never crosses $\xi_a$.
The barrier places every global orbit in the regime where the contraction is
uniform; positive-time compactness and energy dissipation then identify the
remaining scalar limit.

\subsection{Contraction of spatial oscillations}
\label{ssec:oscillation-contraction}

\begin{lemma}\label{lem:oscillation-estimate}
There exists $C_{\mathrm{ctr}}\ge1$, depending only on $\Om$, with the
following property.  For every $G\ge0$ and $v\in H^1(\Om)$ with
$\ave{v}=0$ and $\|\nabla v\|_2\le G$, one has
\begin{equation}\label{eq:nonlinear-oscillation-bound}
\int_\Om(e^v-1)v\,dx
 \le
 C_{\mathrm{ctr}}(1+G^2)
 \exp\left(\frac{G^2}{8\pi}\right)
 \|\nabla v\|_2^2.
\end{equation}
\end{lemma}

\begin{proof}
Since $0\le(e^s-1)s\le s^2e^{|s|}$,
it suffices to estimate $\int v^2e^{|v|}$. Put
$p=1+(1+G^2)^{-1}$ and $p'=G^2+2$,
so that $p$ and $p'$ are conjugate. By H\"older's inequality,
Lemma~\ref{lem:Lr-gradient}, and Lemma~\ref{lem:expint}, with
$X=\|\nabla v\|_2^2$, we obtain
\[
 \int_\Om v^2e^{|v|}\,dx
 \le \|v\|_{2p'}^2
       \left(\int_\Om e^{p|v|}\,dx\right)^{1/p}
 \le C p'X
       \exp\left(\frac{pX}{8\pi}\right).
\]
Here and below, constants depending only on $\Om$ absorb bounded powers of
$C_{\mathrm{MT}}|\Om|$. Since $p'\le2(1+G^2)$ and
$pX\le pG^2\le G^2+1$, after enlarging the constant, we obtain
\eqref{eq:nonlinear-oscillation-bound}.
\end{proof}

With the constant fixed in Lemma~\ref{lem:oscillation-estimate}, define
\begin{equation}\label{eq:eps-osc}
 \eps_{\mathrm{osc}}(G,\overline{m})
 =C_{\mathrm{ctr}}(1+G^2)
  \exp\left(\overline{m}+\frac{G^2}{8\pi}\right).
\end{equation}

\begin{theorem}\label{thm:oscillation-contraction}
Let $u$ be a solution of \eqref{eq:P} on $[0,T)$ and suppose that
\begin{equation}\label{eq:oscillation-bounds}
 \sup_{0\le t<T}\|\nabla u(t)\|_2\le G,
 \qquad
 \sup_{0\le t<T}\ave{u(t)}\le \overline{m}.
\end{equation}
If $\eps>\eps_{\mathrm{osc}}(G,\overline{m})$ then, 
with $\gamma=(\eps-\eps_{\mathrm{osc}}(G,\overline{m}))\lambda_2+a>0$,
one has
\begin{equation}\label{eq:oscillation-contraction}
 \|u^\perp(t)\|_2^2
 \le
 e^{-2\gamma(t-s)}
 \|u^\perp(s)\|_2^2
 \qquad(0\le s\le t<T).
\end{equation}
In particular, if $T=\infty$, then the spatial oscillation decays
exponentially in $L^2(\Om)$.
\end{theorem}

\begin{proof}
Set $X(t)=\|\nabla u(t)\|_2^2$ and $Y(t)=\|u^\perp(t)\|_2^2$.  Positive-time
regularity from Proposition~\ref{prop:lwp} allows the following calculation on
compact subintervals of $(0,T)$.  Subtracting the spatial mean from the
equation gives
\[
 u^\perp_t-\eps\Delta u^\perp
 =e^{\ave{u}}\bigl(e^{u^\perp}-\ave{e^{u^\perp}}\bigr)-au^\perp.
\]
Testing by $u^\perp$ and using $\ave{u^\perp}=0$ yields, for almost every $t\in(0,T)$,
\[
 \frac12Y'+\eps X+aY
 =e^{\ave{u}}\int_\Om(e^{u^\perp}-1)u^\perp\,dx.
\]
By \eqref{eq:oscillation-bounds} and
Lemma~\ref{lem:oscillation-estimate},
\[
 e^{\ave{u}}\int_\Om(e^{u^\perp}-1)u^\perp\,dx
 \le\eps_{\mathrm{osc}}(G,\overline{m})X.
\]
Poincar\'e inequality therefore gives
\[
 \frac12Y'+\gamma Y\le0.
\]
Integration gives \eqref{eq:oscillation-contraction} for
$0<s\le t<T$, and continuity of $u$ in $H^1(\Om)$ extends it to $s=0$.
\end{proof}

\subsection{An energy barrier below the scalar threshold}
\label{ssec:energy-barrier}

The contraction becomes uniform on a global trajectory once its
gradient and mean remain in a bounded region of phase space. The following
barrier provides precisely this control under the sole dynamical condition that
the spatial mean never crosses the scalar threshold.

For $R>0$, set
$C_{\mathrm{en}}(R)=|\Om|+|\Om|^{1/2}R+\frac a2R^2$
and $K_{\mathrm{MT}}=C_{\mathrm{MT}}|\Om|e^{\xi_a}$,
and define
\begin{equation}\label{eq:eps-en}
 \eps_{\mathrm{en}}(R)
 =2C_{\mathrm{en}}(R)
 +2K_{\mathrm{MT}}
   \exp\left(\frac{R^2+1}{8\pi}\right).
\end{equation}

\begin{lemma}\label{lem:energy-barrier}
Let $\|u_0\|_{H^1}\le R$ and suppose that
\begin{equation}\label{eq:mean-below-all-time}
 \ave{u(t)}\le\xi_a
 \qquad(0\le t<\Tmax).
\end{equation}
If $\eps>\eps_{\mathrm{en}}(R)$, then $\Tmax=\infty$ and
\begin{equation}\label{eq:gradient-barrier}
 \sup_{t\ge0}\|\nabla u(t)\|_2^2<R^2+1.
\end{equation}
Moreover,
\begin{equation}\label{eq:Rsharp}
 \sup_{t\ge0}\|u(t)\|_{H^1}
 \le R_*,
\end{equation}
where
\[
 R_*
 =\max\left\{|\Om|^{1/2}\xi_a,R\right\}
 +(1+\lambda_2^{-1/2})\sqrt{R^2+1}.
\]
\end{lemma}

\begin{proof}
Set $X(t)=\|\nabla u(t)\|_2^2$.
By continuity of the solution in $H^1$, $X$ is continuous on
$[0,\Tmax)$. Since
\[
 -F(s)=-e^s+1+s+\frac a2 s^2\le 1+s+\frac a2 s^2,
\]
Lemma~\ref{lem:energy} and the initial $H^1$ bound give
\[
 E_\eps(u(t))
 \le E_\eps(u_0)
 \le\frac\eps2R^2+C_{\mathrm{en}}(R).
\]
On the other hand,
\[
 1+s+\frac a2 s^2\ \ge\ 1-\frac1{2a}\ >\ 0
 \qquad(s\in\R),
\]
so that $F(s)\le e^s$. Hence
\[
 \frac{\eps}{2}X(t)
 =E_\eps(u(t))+\int_\Om F(u(t))\,dx
 \le\frac\eps2R^2+C_{\mathrm{en}}(R)
      +\int_\Om e^{u(t)}\,dx.
\]
By \eqref{eq:Lq} with $q=1$ and
\eqref{eq:mean-below-all-time},
\[
 \int_\Om e^{u(t)}\,dx
 \le C_{\mathrm{MT}}|\Om|
      \exp\left(\xi_a+\frac{X(t)}{8\pi}\right)
 =K_{\mathrm{MT}}e^{X(t)/(8\pi)}.
\]
Therefore
\begin{equation}\label{eq:energy-barrier-ineq}
 \frac{\eps}{2}X(t)
 \le
 \frac\eps2R^2+C_{\mathrm{en}}(R)
 +K_{\mathrm{MT}}e^{X(t)/(8\pi)}.
\end{equation}

Define
\[
 \Theta(d)
 =\frac\eps2(d-R^2)
  -C_{\mathrm{en}}(R)-K_{\mathrm{MT}}e^{d/(8\pi)}.
\]
The strict inequality $\eps>\eps_{\mathrm{en}}(R)$ implies
$\Theta(R^2+1)>0$. By continuity, choose
$d_*\in(R^2,R^2+1)$ with $\Theta(d_*)>0$. Since
$X(0)\le R^2<d_*$, if $X$ reached $d_*$, then at its first hitting time
\eqref{eq:energy-barrier-ineq} would give $\Theta(d_*)\le0$, a
contradiction. Thus
\begin{equation}\label{eq:D-dstar}
 \sup_{0\le t<\Tmax}X(t)\le d_*<R^2+1.
\end{equation}

It remains to control the mean from below. Let $y$ solve the scalar equation
\eqref{eq:ODE} with $y(0)=\ave{u_0}$. By
Lemma~\ref{lem:mean}, $\ave{u(t)}\ge y(t)$. The scalar phase portrait shows
\[
 y(t)\ge\min\{\ave{u_0},0\}
 \ge-\frac R{\sqrt{|\Om|}}.
\]
Together with \eqref{eq:mean-below-all-time}, this gives
\[
 -\frac R{\sqrt{|\Om|}}
 \le\ave{u(t)}\le\xi_a.
\]
Hence, by Poincar\'e inequality and \eqref{eq:D-dstar},
\[
 \begin{aligned}
 \|u(t)\|_{H^1}
 &\le\|u(t)\|_2+\|\nabla u(t)\|_2\\
 &\le |\Om|^{1/2}|\ave{u(t)}|
     +(1+\lambda_2^{-1/2})\|\nabla u(t)\|_2\\
 &\le R_*.
 \end{aligned}
\]
The continuation criterion \eqref{eq:blowup-alt} now excludes finite-time
breakdown. This proves global existence, \eqref{eq:gradient-barrier}, and
\eqref{eq:Rsharp}.
\end{proof}

Define the sufficient dynamical threshold
\begin{equation}\label{eq:eps-dyn}
 \eps_{\mathrm{dyn}}(R)
 =\max\left\{
   \eps_{\mathrm{en}}(R),
   \eps_{\mathrm{osc}}\bigl(\sqrt{R^2+1},\xi_a\bigr),
   1
  \right\}.
\end{equation}
By \eqref{eq:eps-en} and \eqref{eq:eps-osc}, there exists
$C=C(a,\Om)>0$ such that
\begin{equation}\label{eq:eps-dyn-size}
 \eps_{\mathrm{dyn}}(R)
 \le C(1+R)^2
      \exp\left(\frac{R^2}{8\pi}\right)
 \qquad(R\ge1).
\end{equation}

\subsection{Compactness and stationary limit sets}\label{ssec:omega}

The energy barrier gives a uniform $H^1$ bound for every global trajectory in
the large-diffusion regime. Positive-time smoothing upgrades this bound to
precompactness in a space carrying the uniform topology, and energy
dissipation then forces every $\omega$-limit point to be stationary.

For the remainder of this section, we define
\[
 \mathcal X=H^1(\Om)\cap C(\overline\Om),
 \qquad
 \|z\|_{\mathcal X}=\|z\|_{H^1}+\|z\|_\infty.
\]

\begin{lemma}\label{lem:positive-compactness}
Let $u$ be a global solution satisfying
\[
 \sup_{t\ge0}\|u(t)\|_{H^1}<\infty.
\]
Then $\{u(t):t\ge1\}$ is relatively compact in $\mathcal X$. Its
$\omega$-limit set is nonempty, compact, and connected in $\mathcal X$, and
consists of stationary solutions of \eqref{eq:P}.
\end{lemma}

\begin{proof}
The uniform $H^1$ bound and \eqref{eq:Lq} with $q=2$ give
\[
 M_f=\sup_{t\ge0}\|f(u(t))\|_2<\infty.
\]
For $t\ge1$ and $\eta\in(0,1)$, restart the mild formula at $t-1$ and write
$u(t)=P_\eta(t)+Q_\eta(t)$,
where
\[
 \begin{aligned}
 P_\eta(t)&=e^{\eps\Delta_N}u(t-1)
 +\int_\eta^1e^{\eps r\Delta_N}f(u(t-r))\,dr,\\
 Q_\eta(t)&=\int_0^\eta e^{\eps r\Delta_N}f(u(t-r))\,dr.
 \end{aligned}
\]
For fixed $\eta>0$, analytic-semigroup smoothing bounds $P_\eta(t)$
uniformly in $H^2(\Om)$, whereas \eqref{eq:semigroup} gives
\[
 \sup_{t\ge1}\|Q_\eta(t)\|_{\mathcal X}
 \le C_\eps M_f(\eta+\sqrt\eta).
\]
Since in two dimensions $H^2(\Om)\hookrightarrow H^1(\Om),\, H^2(\Om)\hookrightarrow C(\overline\Om)$,
the forward orbit is relatively compact in $\mathcal X$.

The same $H^1$ bound and Lemma~\ref{lem:expint} bound the energy from below;
hence $E_\eps(u(t))\downarrow E_*$. The closures of the orbit tails are
nested nonempty compact connected subsets of $\mathcal X$, so their $\omega$-limit set $\omega(u_0)$ has the same properties. If $u(t_n)\to z$ in
$\mathcal X$, then for every $0<\tau<\Tmax(z)$ continuous dependence gives
$u(t_n+\tau)\to\mathcal S(\tau)z$ in $H^1$. Continuity and dissipation of
the energy therefore yield
\[
 E_\eps(\mathcal S(\tau)z)=E_\eps(z),
 \qquad
 \int_0^\tau\|\partial_t\mathcal S(t)z\|_2^2\,dt=0.
\]
Thus the orbit with initial datum $z$ is constant on its maximal interval. The
continuation criterion makes this interval global, and positive-time
regularity gives $-\eps\Delta z=f(z)$. Elliptic regularity then yields
$z\in C^{2,\alpha}(\overline\Om)$ for every $\alpha\in(0,1)$.
\end{proof}

\begin{proposition}
\label{prop:global-convergence}
Let $R>0$, $\eps>\eps_{\mathrm{dyn}}(R)$, and
$\|u_0\|_{H^1}\le R$. If the corresponding solution is global, then
\begin{equation}\label{eq:global-limit}
 u(t)\longrightarrow0
 \quad\text{or}\quad
 u(t)\longrightarrow\xi_a
 \qquad\text{in }\mathcal X.
\end{equation}
\end{proposition}

\begin{proof}
A global solution cannot satisfy $\ave{u(t_0)}>\xi_a$ at any time, because
Proposition~\ref{prop:mean-threshold} applied after restarting at $t_0$ would
force finite-time blow-up. Hence \eqref{eq:mean-below-all-time} holds, and
Lemma~\ref{lem:energy-barrier} yields
\[
 \sup_{t\ge0}\|u(t)\|_{H^1}\le R_*,
 \qquad
 \sup_{t\ge0}\|\nabla u(t)\|_2<\sqrt{R^2+1}.
\]
Since $\eps>\eps_{\mathrm{osc}}\bigl(\sqrt{R^2+1},\xi_a\bigr)$,
Theorem~\ref{thm:oscillation-contraction} implies that
\begin{equation}\label{eq:global-oscillation-decay}
 \|u(t)-\ave{u(t)}\|_2\longrightarrow0\quad\text{exponentially as}
 \quad t\to\infty.
\end{equation}

Lemma~\ref{lem:positive-compactness} applies to the bounded global orbit.  If
$z\in\omega(u_0)$ and $u(t_n)\to z$ in $\mathcal X$, then
\eqref{eq:global-oscillation-decay} implies that $z$ is spatially constant.
The same lemma shows that $z$ is stationary, so $f(z)=0$. We obtain $\omega(u_0)\subset\{0,\xi_a\}$.
The $\omega$-limit set is nonempty and connected, hence it is either $\{0\}$
or $\{\xi_a\}$.  Relative compactness of the positive orbit then implies
convergence in $\mathcal X$, proving \eqref{eq:global-limit}.
\end{proof}

\subsection{The global trichotomy}\label{ssec:trichotomy}
Proposition~\ref{prop:global-convergence} provides the asymptotic classification of all global trajectories; it remains to obtain exponential decay to $0$, characterize finite-time blow-up by mean crossing, and prove openness of the two alternatives.

\begin{proof}[Proof of Theorem~\ref{thm:intro-phase}]
Every global solution is classified by
Proposition~\ref{prop:global-convergence}. If $u(t)\to0$ in $\mathcal X$,
choose $T>0$, $\mu>0$, and $M\in(0,\xi_a)$ so that
\[
 -\mu<u(x,T)<M
 \qquad(x\in\overline\Om).
\]
Lemma~\ref{lem:invariance} gives exponential decay in $L^\infty$ from time
$T$ onward, and the one-step Duhamel estimate used in
Theorem~\ref{thm:stab} gives exponential decay in $H^1$. Together with
Proposition~\ref{prop:global-convergence} and the finite-time alternative,
this gives the three alternatives in the theorem.

If $\ave{u(t_0)}>\xi_a$ for some $t_0<\Tmax$, then
Proposition~\ref{prop:mean-threshold}, restarted at $t_0$, gives finite-time
blow-up. Conversely, suppose that
$\ave{u(t)}\le\xi_a$ throughout the maximal existence interval. Then
Lemma~\ref{lem:energy-barrier} gives a uniform $H^1$ bound and hence, by
\eqref{eq:blowup-alt}, global existence. This proves
\eqref{eq:intro-mean-crossing}.

It remains to prove openness. Let $u_0$ belong to the blow-up set. By
\eqref{eq:intro-mean-crossing}, there is
$T<\Tmax(u_0)$ such that $\ave{u(T)}>\xi_a$. The time-$T$ map is continuous
in $H^1$ by Proposition~\ref{prop:lwp}, and the mean is continuous on
$L^2$. Thus the same strict mean crossing occurs for all initial data in a
sufficiently small $H^1$ neighborhood of $u_0$; those solutions also blow
up. Hence the blow-up set is relatively open.

Now let $u_0$ belong to the decay set. Choose $T>0$ and numbers
$\mu>0$, $M\in(0,\xi_a)$ such that
\[
 -\mu<u(x,T)<M
 \qquad(x\in\overline\Om).
\]
By Proposition~\ref{prop:lwp}, the strict pointwise
inequalities persist at time $T$ for all initial data in a sufficiently
small $H^1$ neighborhood of $u_0$. Lemma~\ref{lem:invariance} then gives
global exponential decay for every such perturbed solution. The decay set
is therefore relatively open as well. Finally, the bound \eqref{eq:intro-dyn-size} 
is precisely \eqref{eq:eps-dyn-size}.
\end{proof}

The next corollary shows that the threshold alternative is nontrivial. Spatially nonconstant trajectories converging to $\xi_a$ exist.

\begin{corollary}\label{cor:threshold-nontrivial}
Assume $R>\xi_a|\Om|^{1/2}$ and $\eps>\eps_{\mathrm{dyn}}(R)$.
Then the third alternative in Theorem~\ref{thm:intro-phase} is realized by
spatially nonconstant initial data. More precisely, let
$\phi\in C^\infty(\overline\Om)$ be nonconstant and satisfy
$\ave{\phi}=0,\, \|\phi\|_\infty\le1$.
Let $\eta>0$ be any small number. For $0\le s\le2\eta$, let
\[
  u_0^s=\xi_a-s+\eta\phi.
\] Then there exists a unique
$s_\eta\in(0,2\eta)$ such that the solution with initial datum
$ u_0^{s_\eta}=\xi_a-s_\eta+\eta\phi$ is global and satisfies
\[
 u(t)\longrightarrow\xi_a
 \qquad\text{in }\mathcal X
 \quad(t\to\infty).
\]
In addition, along the family $u_0^s$,
the solution blows up in finite time for $0\le s<s_\eta$, whereas it is
global and converges exponentially to $0$ in $\mathcal X$ for
$s_\eta<s\le2\eta$.
\end{corollary}

\begin{proof}
Choose $\eta>0$ sufficiently small that
\[
 \eta<\frac{\xi_a}{3},
 \qquad
 |\Om|\xi_a^2+\eta^2\|\phi\|_{H^1}^2<R^2.
\]
Since $\ave{\phi}=0$, for $0\le s\le2\eta$,
$\|u_0^s\|_{H^1}^2
 =|\Om|(\xi_a-s)^2+\eta^2\|\phi\|_{H^1}^2<R^2$.
Thus the entire family lies in the $H^1$ ball on which
Theorem~\ref{thm:intro-phase} applies.

At $s=0$, the datum $u_0^0=\xi_a+\eta\phi$ is nonconstant and has mean
$\xi_a$. Hence Proposition~\ref{prop:mean-threshold} implies finite-time
blow-up. At the other endpoint,
\[
 0<\xi_a-3\eta
 \le u_0^{2\eta}
 \le\xi_a-\eta<\xi_a,
\]
so Lemma~\ref{lem:invariance} implies global existence and exponential
convergence to $0$.

Let $\mathcal B$ and $\mathcal D$ denote, respectively, the sets of
$s\in[0,2\eta]$ for which the corresponding solution blows up or converges
to $0$. By Theorem~\ref{thm:intro-phase}, both sets are relatively open.
They are disjoint and nonempty, since $0\in\mathcal B$ and
$2\eta\in\mathcal D$. Connectedness of $[0,2\eta]$ therefore yields at
least one parameter outside $\mathcal B\cup\mathcal D$, and the trichotomy
forces the corresponding solution to converge to $\xi_a$.

It remains to prove uniqueness. Suppose that $0\le s_1<s_2\le2\eta$
both generate solutions converging to $\xi_a$, and write
$u_i=u(\cdot,t;u_0^{s_i})$. Since
$u_0^{s_1}-u_0^{s_2}=s_2-s_1>0$, comparison gives $u_1\ge u_2$.
The difference $z=u_1-u_2$ satisfies
\[
 z_t-\eps\Delta z=c(x,t)z,
 \qquad
 c(x,t)=\int_0^1
 f'\bigl(u_2+\theta(u_1-u_2)\bigr)\,d\theta.
\]
Since $u_1,u_2\to\xi_a$ uniformly and
$f'(\xi_a)>0$, there is $T>0$ such that
$c(x,t)\ge f'(\xi_a)/2$ for $t\ge T$. Hence, with
$Z(t)=\int_\Om z(x,t)\,dx>0$,
\[
 Z'(t)\ge\frac{f'(\xi_a)}2 Z(t)
 \qquad(t\ge T),
\]
which is incompatible with $z(t)\to0$ uniformly. Thus the transition
parameter is unique. Since the complement of this single parameter has two
connected components containing $0$ and $2\eta$, respectively, the final
classification follows.
\end{proof}

\section{Sharp exponential scale for uniform stabilization}\label{sec:sharp}

We now optimize the upper stabilization bound at the $L^1$ endpoint and then
construct a matching obstruction by boundary concentration.  The two
arguments identify the same leading exponential scale.

\subsection{Endpoint upper bound}\label{ssec:endpoint}

A fixed $q>1$ would retain the factor $q$ in the Moser--Trudinger inequality and
therefore miss the sharp scale. Letting $q$ approach the endpoint $1$ at the rate
$R^{-2}$ recovers the coefficient $1/(8\pi)$. 

For $R\ge1$, set \[q_R=1+R^{-2}.\]

\begin{corollary}\label{cor:endpoint}
Fix $\delta\in(0,\xi_a)$. There exist
$C_\delta>0$ and $R_\delta'\ge1$, depending only on $a$, $\Om$, and
$\delta$, such that
\begin{equation}\label{eq:eps-stab-polynomial}
 \eps_{\mathrm{stab}}^{(q_R)}(R,\delta)
 \le C_\delta(1+R)^3
      \exp\left(\frac{R^2}{8\pi}\right)
 \qquad(R\ge R_\delta').
\end{equation}
In particular,
\begin{equation}\label{eq:eps-unif-upper}
 \eps_{\mathrm{unif}}(R,\delta)
 \le C_\delta(1+R)^3
      \exp\left(\frac{R^2}{8\pi}\right)
 \qquad(R\ge R_\delta').
\end{equation}
\end{corollary}

\begin{proof}
For fixed $\delta\in(0,\xi_a)$, recall the definition \eqref{eq:S-R-delta}
\[
 S_{R,\delta}\le C_\delta(1+\log R),
 \qquad
 \frac{q_R}{q_R-1}=R^2+1,
 \qquad
 1-\frac1{q_R}=\frac1{R^2+1}.
\]
As the factor $S_{R,\delta}^{1/(R^2+1)}$ is bounded for $R \ge 1$, we have
$\mathcal{H}_{q_R}(S_{R,\delta})=S_{R,\delta}+(R^2+1)S_{R,\delta}^{1/(R^2+1)}
 \le C_\delta R^2$.
Moreover, $q_RD_R=R^2+2+R^{-2},\,d_R^{-1}=\sqrt{R^2+1}+R\le1+2R$.
Formula \eqref{eq:FqR} then gives
\[
 \mathcal F_{q_R,R}
 \le C_{a,\Om}
      \exp\left(\frac{R^2}{8\pi}\right)
 \qquad(R\ge1).
\]
Substituting these estimates into \eqref{eq:eps-stab} proves
\eqref{eq:eps-stab-polynomial} for all sufficiently large $R$.
Theorem~\ref{thm:stab} and the definition \eqref{eq:eps-unif} then imply
\eqref{eq:eps-unif-upper}.
\end{proof}

The lower bound is produced by data concentrated at a boundary point. The
same half-space geometry responsible for the sharp mean-zero
Moser--Trudinger coefficient reappears in the construction.

\subsection{Boundary localization and a Kaplan blow-up criterion}
\label{ssec:boundary-cutoff}
We begin with a boundary-localized version of Kaplan's blow-up argument \cite{Kaplan}.
Our goal is to construct a test function supported in a boundary ball.
Fix $x_0\in\partial\Om$ and choose a Fermi coordinate chart
\[
 \Psi:\overline{B_{2r_0}^+}\longrightarrow\overline\Om ,
\]
where $B_r^+=\{y\in\R^2:|y|<r,\ y_2>0\}$, which is a diffeomorphism onto its image, satisfies $\Psi(0)=x_0$, and maps
$\{y_2=0\}$ into $\partial\Om$. We
use arclength along the boundary and inward normal distance, so the
coordinate vectors are orthonormal at $y=0$.  Write
\[
 \mathfrak j(y)=|\det D\Psi(y)|,
 \qquad
 \mathsf A(y)=\mathfrak j(y)D\Psi(y)^{-1}D\Psi(y)^{-\mathsf{T}}.
\]
After reducing $r_0$ if necessary,
\begin{equation}\label{eq:chart-metric}
 \mathsf A(y)=I+O(|y|),
 \qquad
 \mathfrak j(y)=1+O(|y|),
 \qquad
 \mathsf A_{12}(y)=0.
\end{equation}

\begin{lemma}\label{lem:boundary-cutoff}
Let $\zeta(z)=(1-|z|^2)^2$ for $z\in B_1^+$.
For $0<\rho\le r_0$, set $U_\rho=\Psi(B_\rho^+)$ and define
\[
 \varphi_\rho(x)=
 \begin{cases}
  \mathcal N_\rho^{-1}
  \zeta\!\left(\Psi^{-1}(x)/\rho\right),&x\in U_\rho,\\
  0,&x\notin U_\rho,
 \end{cases}
\]
where
$\mathcal N_\rho
=\int_{U_\rho}\zeta(\Psi^{-1}(x)/\rho)\,dx$.
Then $\varphi_\rho\in W^{2,\infty}(\Om)$,
$\varphi_\rho\ge0$, and $\int_\Om\varphi_\rho\,dx=1$.
Moreover, $\partial_\nu\varphi_\rho=0$ on $\partial\Om$ in the trace sense,
and there exists $\kappa>0$, independent of $\rho$, such that
\begin{equation}\label{eq:weight-laplacian}
 \Delta\varphi_\rho
 \ge-\kappa\rho^{-2}\varphi_\rho
 \qquad\text{a.e.\ and distributionally in }\Om.
\end{equation}
\end{lemma}

\begin{proof}
On the artificial semicircle $\{|z|=1,z_2\ge0\}$, both $\zeta$ and
$\nabla\zeta$ vanish. Thus the zero extension across this interface is
locally the function
$(1-|z|^2)_+^2$.
The function $(1-|z|^2)_+^2$ is globally $C^{1,1}$ on $\R^2$ 
including at the two points where the artificial
semicircle meets $\partial\Om$. Hence its restriction to the chart
belongs to $W^{2,\infty}$, and no interface measure appears in its second
distributional derivatives.

On the flat part of the boundary,
$\partial_{z_2}\zeta(z_1,0)=0$.
Since Fermi coordinates are orthogonal and the $z_2$ direction is normal to
the physical boundary, this implies
$\partial_\nu\varphi_\rho=0$ in the trace sense.

The change of variables $x=\Psi(\rho z)$ gives
\[
 \mathcal N_\rho
 =\rho^2\int_{B_1^+}\zeta(z)\mathfrak j(\rho z)\,dz.
\]
After reducing $r_0$, the Jacobian is bounded above and below by positive
constants. Thus
\[ 
 c_\Psi\rho^2
 \le\mathcal N_\rho
 \le C_\Psi\rho^2
 \qquad(0<\rho\le r_0).
\]

For a smooth function $w$ in the chart, the Laplacian has divergence
form
\[
 (\Delta_x(w\circ\Psi^{-1}))\circ\Psi
 =\frac1{\mathfrak j}
   \partial_{y_i}\!\left(\mathsf A_{ij}\partial_{y_j}w\right).
\]
With $y=\rho z=\Psi^{-1}(x)$, this yields
\begin{equation}\label{eq:scaled-laplacian}
 \rho^2\Delta_x
 \left[\zeta\!\left(\frac{\Psi^{-1}(x)}\rho\right)\right]
 =\mathcal L_\rho\zeta(z),
\end{equation}
where
\[
 \mathcal L_\rho
 =a_\rho^{ij}(z)\partial_{ij}
  +\rho b_\rho^i(z)\partial_i,
\]
and, uniformly for $0<\rho\le r_0$,
\[
 \|a_\rho-I\|_{L^\infty(B_1^+)}\le C_\Psi\rho,
 \qquad
 \|b_\rho\|_{L^\infty(B_1^+)}\le C_\Psi.
\]
Indeed, the fixed Fermi chart is smooth, so $\mathfrak j$, $\mathsf A$,
and their first derivatives are uniformly bounded on the chart. The first
estimate follows from \eqref{eq:chart-metric}, while these derivative bounds
control the first-order coefficients after scaling and give the second.

Since
$\Delta\zeta=-8+16|z|^2$,
choose $\sigma\in(2^{-1/2},1)$ and set
$c_\sigma=16\sigma^2-8>0$. On
$\{\sigma\le|z|\le1,z_2\ge0\}$,
$\Delta\zeta\ge c_\sigma$. Moreover,
\[
 |\mathcal L_\rho\zeta-\Delta\zeta|
 \le C_\Psi\rho
 \bigl(\|D^2\zeta\|_\infty+\|\nabla\zeta\|_\infty\bigr).
\]
Reducing $r_0$ once more, we may arrange that the right-hand side is at most
$c_\sigma/2$. Hence
$\mathcal L_\rho\zeta\ge c_\sigma/2$ on this annulus. On
$\{|z|\le\sigma,z_2\ge0\}$, the function $\mathcal L_\rho\zeta$ is bounded
below uniformly in $\rho$, whereas
$\zeta\ge(1-\sigma^2)^2$. It follows that
\[
 \mathcal L_\rho\zeta\ge-\kappa\zeta
 \qquad\text{on }B_1^+.
\]
Scaling in \eqref{eq:scaled-laplacian} and normalization by
$\mathcal N_\rho$ prove \eqref{eq:weight-laplacian} almost everywhere. Since
$\varphi_\rho\in W^{2,\infty}(\Om)$, the same inequality holds
distributionally.
\end{proof}

\begin{lemma}\label{lem:localized-Kaplan}
Let $0<\rho\le r_0$ and $h\ge1$ satisfy
$1+ah\le e^h/8$.
Suppose that $u_0\in H^1(\Om)$ satisfies
$u_0\ge0$ and $u_0\ge h$ a.e.\ on $U_\rho$.
If
\[
 0<\eps
 \le
 \frac{\rho^2}{8\kappa}
 \frac{e^h}{h},
\]
then the corresponding solution blows up in finite time. More precisely,
\[
 \Tmax\le2e^{-h}.
\]
\end{lemma}

\begin{proof}
Let $\varphi_\rho$ be the cutoff from
Lemma~\ref{lem:boundary-cutoff} and set
\[
 \mathcal J(t)
 =\int_\Om u(x,t)\varphi_\rho(x)\,dx.
\]
Since $\varphi_\rho\,dx$ is a probability measure supported in $U_\rho$,
$\mathcal J(0)\ge h$.
Moreover, $u_0\ge0$ and comparison with the zero solution give
$u\ge0$ throughout the existence interval.

For $t>0$,  we deduce from the regularity  of $u$ and $\varphi_\rho$ that
\[\int_\Om(\Delta u)\varphi_\rho\,dx
 =-\int_\Om\nabla u\cdot\nabla\varphi_\rho\,dx
 =\int_\Om u\,\Delta\varphi_\rho\,dx.\]

Using $u\ge0$, \eqref{eq:weight-laplacian}, and Jensen's inequality
with respect to the probability measure $\varphi_\rho\,dx$, we obtain
\[
  \eps\int_\Om(\Delta u)\varphi_\rho\,dx
  \ge -\eps\kappa\rho^{-2}\mathcal J,\quad
  \int_\Om e^u\varphi_\rho\,dx
  \ge e^{\mathcal J}.
\]
Thus, with $\Gamma=\eps\kappa\rho^{-2}$,
\[
 \begin{aligned}
  \mathcal J'
  &=\eps\int_\Om(\Delta u)\varphi_\rho\,dx
    +\int_\Om e^u\varphi_\rho\,dx-1-a\mathcal J\\
  &\ge e^{\mathcal J}-1-(a+\Gamma)\mathcal J.
 \end{aligned}
\]
The function $\mathcal J$ is continuous on $[0,\Tmax)$ and $C^1$ on
$(0,\Tmax)$. Integrating the preceding inequality over $[\tau,t]$ and
letting $\tau\downarrow0$ shows that the corresponding integral
inequality is valid from time zero.

By the assumption on $\eps$, $\Gamma h\le e^h/8$.
Set $K=a+\Gamma$; then $1+Kh\le e^h/4$.
The function $g(s)=(1+Ks)e^{-s}$ is decreasing for $s\ge1$, since
$g'(s)=e^{-s}(K-1-Ks)<0$. Therefore
$e^s-1-Ks\ge e^s/2$ for $s\ge h$.

Let $\mathcal{Y}$ be the maximal solution of
$\mathcal{Y}'=e^{\mathcal{Y}}-1-K\mathcal{Y}$, $\mathcal{Y}(0)=h$.
The scalar comparison principle gives $\mathcal J(t)\ge \mathcal{Y}(t)$
while both functions are finite. 
The preceding estimate gives $\mathcal Y'\ge e^{\mathcal Y}/2$ and hence
\[
 \frac{d}{dt}e^{-\mathcal Y(t)}\le-\frac12 ,
 \quad\text{so}\quad
 e^{-\mathcal Y(t)}\le e^{-h}-\frac t2 .
\]
Therefore the maximal existence time of $\mathcal Y$ is at most $2e^{-h}$.
If the PDE existed beyond that time, $\mathcal J$ would remain finite
while $\mathcal{Y}(t)\to\infty$, a contradiction. 
Hence $\Tmax\le2e^{-h}$ as claimed.
\end{proof}

\subsection{Boundary Moser profiles and lower bounds}
\label{ssec:boundary-moser}

\begin{lemma}\label{lem:boundary-moser-profile}
Fix $r\in(0,r_0]$. For $L\ge1$, set $\rho_L=re^{-L}$
and define $\chi_L:\overline{B_{r}^+}\to[0,1]$ by
\[
 \chi_L(y)=
 \begin{cases}
  1,&|y|\le\rho_L,\\
  L^{-1}\log(r/|y|),&\rho_L<|y|\le r.
 \end{cases}
\]
Let $\psi_L(\Psi(y))=\chi_L(y)$ on $\Psi(B_{r}^+)$ and extend it
by zero outside this set.  Then
$\psi_L\in H^1(\Om)\cap L^\infty(\Om)$, $0\le\psi_L\le1$, and
$\psi_L=1$ on $\Psi(B_{\rho_L}^+)$. In addition,
\begin{equation}\label{eq:boundary-Moser-estimates}
 \|\nabla\psi_L\|_2^2=\frac\pi L+O(L^{-2}),
 \qquad
 \|\psi_L\|_2^2=O(L^{-2}),
 \qquad
 \int_\Om\psi_L\,dx\le\frac{C_\Psi^*r^2}{L},
\end{equation}
where $C_\Psi^*>0$ depends only on the fixed chart and the estimates are
uniform for $0<r\le r_0$.
\end{lemma}

\begin{proof}
Since $\chi_L$ has zero trace on the artificial boundary $|y|=r$, its
zero extension belongs to $H^1(\Om)$.  In the flat half-disc,
\[
 \int_{B_{r}^+}|\nabla\chi_L|^2\,dy
 =\frac\pi{L^2}\int_{\rho_L}^{r}\frac{ds}{s}
 =\frac\pi L.
\]
The coefficient is $\pi$, rather than the interior value $2\pi$, because
only a half-disc is involved.  From \eqref{eq:chart-metric}, the metric error
is
\[
 O\left(\int_{B_r^+}|y||\nabla\chi_L|^2\,dy\right)=O(L^{-2}).
\]
For the remaining estimates, the change of variables
$t=\log(r/s)$ and the uniform bounds on the Jacobian give
\[
 \begin{aligned}
  \|\psi_L\|_2^2
  &\le C_\Psi r^2\left(
       e^{-2L}
       +\frac1{L^2}\int_0^L t^2e^{-2t}\,dt
      \right)
   \le \frac{C_\Psi r^2}{L^2}
   =O(L^{-2}),\\
  \int_\Om\psi_L\,dx
  &\le C_\Psi r^2\left(
       e^{-2L}
       +\frac1L\int_0^L te^{-2t}\,dt
      \right)
   \le \frac{C_\Psi^*r^2}{L}.
 \end{aligned}
\]
This proves the remaining estimates in
\eqref{eq:boundary-Moser-estimates}.
\end{proof}

As a first consequence, the construction already shows
that the uniform stabilization threshold is strictly positive at every
radius.

\begin{corollary}\label{cor:eps-unif-positive}
For every $R>0$ and $\delta\in(0,\xi_a)$,
$\eps_{\mathrm{unif}}(R,\delta)>0$.
\end{corollary}

\begin{proof}
Fix $R>0$ and $\delta\in(0,\xi_a)$. Choose $h\ge1$ so large that $1+ah\le e^h/8$, and then choose
$r\in(0,r_0]$ so small that
$hC_\Psi^*r^2/|\Om|<\xi_a-\delta$.
By Lemma~\ref{lem:boundary-moser-profile},
$\|\psi_L\|_{H^1}\to0$ as $L\to\infty$.
Hence, for all sufficiently large $L$, $u_0=h\psi_L$ satisfies
$\|u_0\|_{H^1}\le R$. Moreover,
$\ave{u_0}\le hC_\Psi^*r^2/(L|\Om|)<\xi_a-\delta$,
so $u_0\in\mathcal A(R,\delta)$.

Set $\rho=re^{-L}$.
Since $u_0=h$ on $U_\rho$, Lemma~\ref{lem:localized-Kaplan} shows that
the solution with initial datum $u_0$ blows up whenever
\[
 0<\eps
 \le
 \frac{r^2e^{h-2L}}{8\kappa h}.
\]
The right-hand side is positive. Thus stabilization fails at some
positive diffusivity, and therefore $\eps_{\mathrm{unif}}(R,\delta)>0$.
\end{proof}

We now let the height and concentration scale grow with $R$ and optimize
their competition to obtain the matching exponential lower bound.

\begin{theorem}\label{thm:necessity}
Fix $\delta\in(0,\xi_a)$. There exist $c_\delta,R_\delta''>0$, depending
only on $a$, $\Om$, and $\delta$, such that
\begin{equation}\label{eq:lowerbound}
 \eps_{\mathrm{unif}}(R,\delta)
 \ge c_\delta R^{-2}
      \exp\left(\frac{R^2}{8\pi}\right)
 \qquad(R\ge R_\delta'').
\end{equation}
\end{theorem}

\begin{proof}
For fixed $\delta$, choose $r=r(\delta)\in(0,r_0]$ so small that
$5C_\Psi^*r^2/|\Om|<\xi_a-\delta$.
We next choose the concentration scale by optimizing the leading
reaction--diffusion competition. If $h_{R,L}=R/\|\psi_L\|_{H^1}$ and $L\asymp R^2$, then
Lemma~\ref{lem:boundary-moser-profile} gives
$h_{R,L}=R\sqrt{L/\pi}+O(1)$.
Lemma~\ref{lem:localized-Kaplan} applies to a concentration core of radius $\rho=re^{-L}$ 
for diffusivities up to order
\[
 \rho^2\frac{e^{h_{R,L}}}{h_{R,L}}
 =
 \frac{r^2}{h_{R,L}}
 \exp\left(h_{R,L}-2L\right).
\]
The leading exponent is optimized by
\[
  \max_{L>0}
  \left(
    R\sqrt{\frac{L}{\pi}}-2L
  \right)
  =
  \frac{R^2}{8\pi},
  \quad\text{attained at }
  L=\frac{R^2}{16\pi}.
\]
Accordingly, set
\[
 L_R=\frac{R^2}{16\pi},
 \qquad
 h_R=\frac{R}{\|\psi_{L_R}\|_{H^1}},
 \qquad
 u_0^R=h_R\psi_{L_R}.
\]
Then $\|u_0^R\|_{H^1}=R$. Lemma~\ref{lem:boundary-moser-profile} yields
\[
 h_R=\frac{R^2}{4\pi}+O(1),
 \qquad
 \frac{h_R}{L_R}=4+O(R^{-2}).
\]
Consequently, for all sufficiently large $R$,
\[
 \ave{u_0^R}
 \le
 \frac{h_R}{L_R}
 \frac{C_\Psi^*r^2}{|\Om|}
 \le
 \frac{5C_\Psi^*r^2}{|\Om|}
 <\xi_a-\delta.
\]
Thus $u_0^R\in\mathcal A(R,\delta)$.

Set $\rho_R=re^{-L_R}$. For all sufficiently large $R$, one also has
$h_R\ge1$ and $1+ah_R\le e^{h_R}/8$.
Lemma~\ref{lem:localized-Kaplan} therefore shows that the solution with
initial datum $u_0^R$ blows up whenever
\[
 0<\eps
 \le
 \eps_{\mathrm{lb}}(R,\delta)
 =
 \frac{r^2}{8\kappa}
 \frac{\exp(h_R-2L_R)}{h_R}.
\]
Finally,
\[
 h_R-2L_R=\frac{R^2}{8\pi}+O(1),
 \qquad
 h_R\asymp R^2.
\]
Hence, after increasing $R_\delta''$ if necessary,
\[
 \eps_{\mathrm{lb}}(R,\delta)
 \ge
 c_\delta R^{-2}
 \exp\left(\frac{R^2}{8\pi}\right).
\]
Since stabilization fails at $\eps=\eps_{\mathrm{lb}}(R,\delta)$, the definition
of $\eps_{\mathrm{unif}}$ gives
$\eps_{\mathrm{unif}}(R,\delta)\ge\eps_{\mathrm{lb}}(R,\delta)$,
which proves \eqref{eq:lowerbound}.
\end{proof}

We are now ready to prove Theorem~\ref{thm:intro-main}.

\begin{proof}[Proof of Theorem~\ref{thm:intro-main}]
Corollary~\ref{cor:eps-unif-positive} gives
$\eps_{\mathrm{unif}}(R,\delta)>0$ for every $R>0$. On the other hand,
Theorem~\ref{thm:stab} with $q=2$ gives
$\eps_{\mathrm{unif}}(R,\delta)
\le\eps_{\mathrm{stab}}^{(2)}(R,\delta)<\infty$
for $R>0$, so the first assertion holds for all radii.

For the asymptotic bounds, let
$R_\delta=\max\{R_\delta',R_\delta''\}$.
Corollary~\ref{cor:endpoint} and Theorem~\ref{thm:necessity} then give
\eqref{eq:intro-two-sided} for $R\ge R_\delta$. Taking logarithms in \eqref{eq:intro-two-sided} yields the stated asymptotic formula.
\end{proof}

\section{Concluding remarks and open problems}\label{sec:conclusion}
The two main results point to a simple idea. Large diffusion must homogenize an $H^1$-bounded profile before the exponential reaction amplifies its spatial concentration. The mean-zero Moser--Trudinger inequality quantifies the cost of suppressing this concentration, while the boundary Moser construction shows that the same boundary-critical geometry produces the worst-case obstruction. 

One may ask how this picture changes under homogeneous Dirichlet conditions. The Dirichlet heat flow drives its linear dynamics toward zero, and the spatial mean no longer obeys a diffusion-independent scalar comparison. A first question is whether every bounded $H_0^1(\Om)$ family is uniformly stabilized to zero. The sharp Dirichlet Moser--Trudinger coefficient is $4\pi$ \cite{Moser}, and concentration is of interior rather than boundary type. These features suggest the sharp asymptotic scale
\[
 \log \eps_{\mathrm{unif}}^{\mathrm{D}}(R)
 \overset{?}{=}
 \frac{R^2}{16\pi}+O(\log R)
 \qquad(R\to\infty),
\]
where $\eps_{\mathrm{unif}}^{\mathrm{D}}(R)$ denotes the Dirichlet uniform stabilization threshold.

The higher-dimensional Sobolev-critical setting raises a second question. For $N\ge3$, one may consider the Neumann problem with reaction
\[
 f(u)=(u_+)^{(N+2)/(N-2)}-au,
 \quad u_+=\max\{u,0\},\quad a>0,
\]
which retains the convexity. We do not know whether sufficiently large diffusion again recovers the scalar trichotomy uniformly on bounded $H^1$ families and what determines the sharp cost of uniform subthreshold stabilization. The critical scaling suggests that boundary-concentrating
Sobolev bubbles should play the role of the boundary Moser profiles used in this work. If such bubbles are indeed the extremal obstruction, the uniform stabilization threshold are conjectured to be $R^{4/(N-2)}$.

\end{document}